\documentclass[11pt]{amsart}

\usepackage{subcaption}

\usepackage[margin=1.25in,
           top=1.25in,
           bottom=1.25in]{geometry}
\usepackage{amssymb}
\usepackage{amscd}
\usepackage{mathtools}
\usepackage[all]{xy}
\usepackage{bbm}
\usepackage{mathrsfs}
\usepackage{enumerate}
\usepackage{tikz}
\usepackage{stmaryrd}
\usepackage{etoolbox}
\usepackage{array}
\usepackage{hyperref}

\newcommand{\vc}{\operatorname{VC}}

\newcommand{\Z}{\mathbb{Z}}

\newcommand{\R}{\mathbb{R}}

\newcommand{\inv}{^{-1}}

\newcommand{\eps}{\varepsilon}

\newcommand{\bd}{\partial}

\newcommand{\la}{\langle}
\newcommand{\ra}{\rangle}

\newcommand{\vol}{\operatorname{Vol}}

\newcommand{\area}{\operatorname{Area}}

\theoremstyle{plain}
\newtheorem{theorem}{Theorem}[section]
\newtheorem{corollary}[theorem]{Corollary}
\newtheorem{prop}[theorem]{Proposition}

\newtheorem{conj}[theorem]{Conjecture}

\theoremstyle{definition}
\newtheorem{defn}[theorem]{Definition}
\newtheorem{rem}[theorem]{Remark}

\numberwithin{equation}{section}

\begin{document}

\title[Pairs of PMC Spheres]{Pairs of Embedded Spheres with Pinched Prescribed Mean Curvature}
\author{Liam Mazurowski}
\address{Lehigh University, Department of Mathematics, Bethlehem, Pennsylvania 18015}
\email{lim624@lehigh.edu}

\author{Xin Zhou}
\address{Cornell University, Department of Mathematics, 212 Garden Ave, Ithaca, New York, 14853}
\email{xinzhou@cornell.edu}

\begin{abstract}
Assume $h$ is a positive function on the unit three-sphere which satisfies the pinching condition $h < h_0 \approx 0.547$. We prove the existence of at least two embedded two-spheres with prescribed mean curvature $h$.  The same result holds for sign-changing functions $h$ satisfying $\vert h\vert < h_0$ under a mild assumption on the zero set. 
\end{abstract}

\maketitle

\section{Introduction}

Let $h$ be a smooth function on a manifold $M$.  A hypersurface $\Sigma$ in $M$ is said to have prescribed mean curvature $h$ if $H_\Sigma = h|_\Sigma$.  Prescribed mean curvature hypersurfaces are a generalization of minimal hypersurfaces and constant mean curvature hypersurfaces, and have played a crucial role in recent breakthroughs in min-max theory \cite{wang2023existence},\cite{zhou2020multiplicity} and the study of scalar curvature \cite{chodosh2024generalized},\cite{gromov2020no}.  

Yau \cite[Problem 59]{yau1982problem} asked what conditions on a function $h\colon \R^3\to \R$ are necessary to ensure the existence of a closed surface of prescribed genus with mean curvature given by $h$.  Yau's problem has been studied by a number of authors \cite{caldiroli2004h},\cite{caldiroli2002existence},\cite{caldiroli2011bubbles},\cite{mazurowski2025prescribed},\cite{stryker2024min},\cite{treibergs1983embedded},\cite{yau1997remark}. In these papers, the non-compactness of $\R^3$ leads to many subtle difficulties.   Given this, it is natural to ask a similar question in closed manifolds.  If one does not insist on topological control, the second author and J. Zhu \cite{zhou2020existence} have given an essentially complete answer: a closed, prescribed mean curvature surface exists for every positive function $h$ and for generic sign-changing functions $h$. With topological control, the problem is more difficult and the existence theory is not yet complete, although there are some partial results. For example, Sarnataro and Stryker \cite{sarnataro2023optimal} have proven the existence of an embedded two-sphere with prescribed mean curvature $h$ in the unit three-sphere for generic prescription functions satisfying the pinching condition $\vert h\vert < h_0 \approx 0.547$.  Also, D.R.\ Cheng and the second author \cite{cheng2023existence} have proven that an arbitrary Riemmanian three-sphere contains a branched immersed two-sphere of  constant mean curvature $h\equiv c$ for almost-every $c$.

Beyond this, one can ask how many prescribed mean curvature surfaces exist in a given manifold.  A collection of conjectures postulate the existence of at least two solutions in various circumstances.  Arnold \cite{arnold2004arnold} famously conjectured the existence of two distinct curves with constant geodesic curvature (or more generally, two distinct magnetic geodesics) in any Riemannian two-sphere.  The Twin-Bubble Conjecture \cite{mazurowski2024infinitely},\cite{zhou2022mean} is a higher dimensional analog.  It asserts that in any closed Riemannian manifold $M$ of dimension $3\le n+1 \le 7$ there exist two distinct almost-embedded constant mean hypersurfaces of any given mean curvature. More generally, it is conjectured that for a generic function $h$ on $M$, there are always two distinct almost-embedded hypersurfaces with mean curvature given by $h$.  These conjectures are all still open.

In this paper, we show that under the same pinching condition as Sarnataro-Stryker \cite{sarnataro2023optimal}, there are always at least two distinct embedded two-spheres in $S^3$ of a given prescribed mean curvature. We define $h_0$ to be the real solution to 
\begin{equation}
\label{h0def}
    \pi h_0^3 + 2h_0^2 + 4\pi h_0 - 8 = 0
\end{equation}
and note that $h_0 \approx 0.547$.

\begin{theorem}
\label{main}
Let $S^3$ denote the unit three-sphere.  Assume that $h\colon S^3\to \R$ is admissible for the PMC min-max theory and that $\vert h\vert < h_0$.  Then there exist at least two distinct two-spheres smoothly embedded in $S^3$ with mean curvature given by $h$. 
\end{theorem}

The proof of the theorem uses the Simon-Smith min-max theory for the prescribed mean curvature functional developed in \cite{sarnataro2023optimal} and further refined in \cite{wang2023existence}.  It also relies on the Smale conjecture \cite{hatcher1983proof}, which implies that the space $\mathcal E$ of oriented, smoothly embedded two-spheres in $S^3$ is homotopy equivalent to $S^3$.  Intuitively, this implies that the $\mathcal A^h$ functional should have at least two critical points on $\mathcal E$, corresponding to the two prescribed mean curvature two-spheres detected in Theorem \ref{main}.   

The admissibility hypothesis in the theorem is very mild and holds for every positive function $h$ and for generic sign-changing functions $h$.  As in \cite{sarnataro2023optimal}, the pinching condition on $h$ is used to guarantee that solutions are smooth and embedded.  An essential feature of the theorem is the explicit constant $h_0$, as it should be possible to obtain the same result via perturbation methods when $h$ satisfies $\vert h\vert < \eps_0$ for some small (inexplicit) $\eps_0$. We expect that there should be counter-examples to the existence of even one embedded solution if the pinching condition is dropped entirely; see the ensuing discussion. 

\begin{rem}
    We note that the Smale conjecture has also motivated several recent breakthroughs on the existence of minimal surfaces with prescribed topology in three manifolds with positive Ricci curvature; see \cite{chu2025minimal},\cite{chu2024existence},\cite{chu2025existence},\cite{li2024minimal},\cite{li2024existence},\cite{wang2023existence}.
\end{rem}

\subsection{Motivation} The following is a well-known conjecture of Arnold \cite{arnold2004arnold}.  It has connections with many different areas of mathematics, including symplectic geometry, dynamical systems, and min-max theory.

\begin{conj}
Let $g$ be a Riemannian metric on $S^2$.  Then for any $\kappa > 0$, there are two closed curves in $(S^2,g)$ with constant geodesic curvature $\kappa$. 
\end{conj}

Arnold's conjecture is still open, although there is some progress due to Schneider \cite{schneider2011closed}, who proved the existence of two solutions for every $\kappa >0$ when the metric $g$ has quarter pinched positive curvature.  Schneider \cite{schneider2012alexandrov} also proved the existence of two Alexandrov embedded solutions when the metric is non-negatively curved. (An immersion of the circle is called an Alexandrov embedding if it extends to an immersion of a disk.)  Finally we mention that Rosenberg-Schneider \cite{rosenberg2011embedded} proved that a two-sphere with positive Gaussian curvature has two embedded curves with constant geodesic curvature $\kappa$ for all $\kappa \in [0,\eps_0)$ for some small (inexplicit) constant $\eps_0$. 

The existence theory for even a single embedded solution is not fully understood for arbitrary metrics $g$, although Novikov \cite{novikov1982hamiltonian} conjectures there should always be  a solution.  In this regard, we mention that Sarnataro-Stryker \cite{sarnataro2023existence}, building on the methods of \cite{ketover2023existence} and \cite{zhou2020min}, obtained the existence of a single embedded solution for every $\kappa > 0$ under the assumption that the metric has one-eighth  pinched positive curvature.  For general Riemannian metrics, the best known result \cite{asselle2016lusternik},\cite{cheng2021existence} is that there is a single closed {\it immersed} solution for {\it almost-every} $\kappa > 0$. 

In ambient dimension 3, constant mean curvature surfaces are a natural analog to curves with constant geodesic curvature.  Rosenberg-Smith \cite{rosenberg2020degree} made the following conjecture on the existence of CMC two-spheres in positively curved three-spheres. 

\begin{conj}
For any $c > 0$ and any metric $g$ on $S^3$ with positive sectional curvature, there is an embedding of $S^2$ into $S^3$ with constant mean curvature $c$. 
\end{conj} 

In fact, the conjecture as stated is not true.  Torralbo \cite{torralbo2010rotationally} has classified all of the constant mean curvature two-spheres in the Berger spheres, and it turns out that there is a constant $c > 0$ and a  (positively curved) Berger sphere $(S^3,g_{\text{Ber}})$ for which there is no embedded two-sphere with constant mean curvature $c$.  Nevertheless, every Berger sphere contains an {\it immersed} two-sphere of any given constant mean curvature.  Thus, it is natural to make the following modified conjecture; c.f. \cite{zhou2022mean}. 

\begin{conj}
For any $c > 0$ and any metric $g$ on $S^3$, there is an {\it immersion} of $S^2$ into $S^3$ with constant mean curvature $c$. 
\end{conj}

D.R. Cheng and the second author \cite{cheng2023existence} proved that if $(S^3,g)$ has positive Ricci curvature, then there is a branched immersion of $S^2$ into $S^3$ of any given constant mean curvature.  In fact, when $g$ is an arbitrary metric on $S^3$, they also proved the existence of a branched immersion with constant mean curvature $c$ for almost every $c>0$.   The above discussion highlights the fact that one cannot in general expect to obtain the existence of embedded surfaces with prescribed non-zero mean curvature and controlled topology without some form of pinching condition.   We further remark that, in light of the Arnold conjecture, it is reasonable to conjecture the existence of a second solution. 

\begin{conj}
For any $c > 0$ and any metric $g$ on $S^3$, there are at least two immersions of $S^2$ into $S^3$ with constant mean curvature $c$. 
\end{conj}

The results discussed so far consider simple curvature prescription functions $h\equiv c$ but allow the metric to vary.  It is also interesting to consider the ``dual'' problem where the metric is simple, but the curvature prescription function is allowed to be non-constant.  To this end, Sarnataro-Stryker \cite{sarnataro2023optimal} have considered the problem of finding embedded two-spheres with prescribed mean curvature in the unit three-sphere.  They show that for a generic function $h\colon S^3\to \R$ satisfying the pinching condition $\vert h\vert < h_0$, there is an embedding of $S^2$ in $S^3$ with mean curvature $h$.  Again, by analogy with Arnold's conjecture, it is natural to expect the existence of a second solution under the same hypotheses.  Our main result Theorem \ref{main} confirms that this is indeed the case.  While it is unclear whether the pinching constant $h_0$ is sharp, we believe there should be counter-examples to the existence of embedded solutions if one drops the pinching condition entirely.  On the other hand, we conjecture that there should always be at least two immersed solutions. 

\begin{conj}
For a generic function $h\colon S^3\to \R$ on the unit three-sphere, there should be at least two immersions of $S^2$ into $S^3$ with mean curvature prescribed by $h$. 
\end{conj}

Finally, we mention that if one drops the topological control requirement, there are also similar conjectures \cite{mazurowski2024infinitely, zhou2022mean} on the existence of at least two solutions, in both the constant mean curvature and prescribed mean curvature case.  

\begin{conj}[Twin Bubble Conjecture]
Let $M^{n+1}$ be a closed manifold of dimension $3\le n+1\le 7$.  Then for every $c > 0$ there are at least two almost-embedded hypersurfaces in $M$ with constant mean curvature $c$. 
\end{conj}

\begin{conj}
Let $M^{n+1}$ be a closed manifold of dimension $3\le n+1\le 7$.  For a generic function $h\colon M\to \R$, there at least two almost-embedded hypersurfaces in $M$ with mean curvature prescribed by $h$. 
\end{conj}

The Twin Bubble Conjecture is known to be true when $c$ is very large by the work of Ye \cite{ye1991foliation} and Pacard-Xu \cite{pacard2009constant}.  It is also known when $c$ is very small. In fact, Dey \cite{dey2023existence} proved that the number of CMCs with mean curvature $c$ goes to infinity as $c\to 0$; also see \cite{mazurowski2022cmc} for an explicit construction of some CMCs with small mean curvature.

\subsection{Sketch of Proof} In this section, we sketch the proof of Theorem \ref{main}.  Fix a smooth function $h$ on $S^3$ which is admissible for the PMC min-max theory and satisfies $\vert h\vert<h_0$. By the work of Sarnataro-Stryker \cite{sarnataro2023optimal}, there exists an oriented, embedded two-sphere $\Gamma$ in $S^3$ with mean curvature prescribed by $h$.  The proof of Theorem \ref{main} is by contradiction.  Suppose that $\Gamma$ is the only oriented, embedded two-sphere with mean curvature prescribed by $h$.

 Let $\mathcal E$ be the space of oriented, embedded two-spheres in $S^3$.  By the Smale conjecture, $\mathcal E$ is homotopy equivalent to $S^3$ and the set of all oriented, equatorial two-spheres represents the non-trivial homology class in $H_3(\mathcal E,\Z_2)$.  For any $\Sigma\in \mathcal E$, the orientation determines a unique region $\Omega$ enclosed by $\Sigma$ and we define 
 \[
 \mathcal A^h(\Sigma) = \area(\Sigma) - \int_\Omega h. 
 \]
Next, we apply min-max theory for the $\mathcal A^h$ functional to several different homotopy classes of maps into $\mathcal E$.

First, for each $p\in S^3$, let $\Pi_p$ be the set of sweepouts $\Psi\colon [0,\pi]\to \mathcal E$ that originate at the point $p$ and terminate at the antipodal point $-p$.  We apply Simon-Smith min-max theory for the $\mathcal A^h$ functional to the homotopy class $\Pi_p$.  As in \cite{sarnataro2023optimal}, the pinching condition $\vert h\vert < h_0$ ensures that min-max detects a smooth, embedded two-sphere with mean curvature $h$, which is therefore necessarily given by $\Gamma$.  

Next, we consider a four parameter min-max problem.  Let $X = S^3\times [0,\pi]$ and let $Z = S^3 \times \{0,\pi\}$. Define a map $\Phi \colon X\to \mathcal E$ by setting 
\[
\Phi(p,r) = \bd B(p,r),
\]
where this sphere is oriented by the normal vector pointing toward $p$.    Let $\Pi$ be the set of all maps $\Psi\colon X\to \mathcal E$ which are homotopic to $\Phi$ via a homotopy fixing the values on $Z$.  

We apply the Simon-Smith min-max theory for the $\mathcal A^h$ functional to the homotopy class $\Pi$. Again, by the pinching condition $\vert h\vert < h_0$, this produces a smooth, embedded two-sphere with mean curvature $h$, which must be $\Gamma$.  Assume for simplicity that there is an optimal map $\Psi\colon X\to \mathcal E$ belonging to $\Pi$ with 
\begin{equation}
\label{optimal} 
\sup_{x\in X} \mathcal A^h(\Psi(x)) = \mathcal A^h(\Gamma). 
\end{equation}
Note then that for each fixed $p$, the path $r\mapsto \Psi(p,r)$ is an optimal map belonging to $\Pi_p$.  Therefore, for each $p$ there must be some $r(p)$ such that $\Psi(p,r(p)) = \Gamma$.  Again, assume for simplicity that for each $p\in S^3$ this point $r(p)$ is unique.  Then $Y = \{(p,r(p)): p\in S^3\} \subset X$ is a three-sphere homologous to $S^3 \times \{\frac{\pi}{2}\}$. 

Since $\Psi$ is homotopic to $\Phi$, it follows that $\Psi_*[Y] = \Phi_*[S^3 \times \{\frac {\pi}{2}\}]$ in $H_3(\mathcal E,\Z_2)$. Now observe that $\Psi(p,r) = \Gamma$ for all $(p,r)\in Y$ and so $\Psi_*[Y] = 0$.  On the other hand, $\Phi_*[S^3\times \{\frac{\pi}{2}\}]$ is the family of oriented, equatorial two-spheres, which is non-trivial in $H_3(\mathcal E,\Z_2)$.  This is a contradiction, and the theorem follows. 

There are two technical issues that need to be addressed to turn the above sketch into a rigorous argument.  The first is that the point $r(p)$ need not be unique.  We can handle this by using a Lusternik-Schnirelmann argument (c.f.\ the Marques-Neves \cite{marques2014min} proof of the Willmore conjecture) to produce a suitable three-dimensional complex $Y$ in $X$ which is homologous to the fiber.  The second issue is that there need not be an optimal map achieving equality in \eqref{optimal}.  Because of this, the map $\Psi$ is not literally constant when restricted to $Y$.  Rather, $\Psi$ maps $Y$ into a small neighborhood of $\Gamma$ with respect to a certain weak topology. To handle this, we prove an interpolation theorem which essentially says that the space $\mathcal E$ is semi-locally contractible with respect to this weak topology.  The interpolation theorem follows from combining the Smale conjecture with a filigree retraction lemma of Ketover-Liokumovich \cite{ketover2025smale}.

\subsection{Organization} The remainder of the paper is organized as follows.  In Section \ref{preliminary}, we discuss the tools needed to prove our main theorem, including the Smale conjecture, the PMC Simon-Smith min-max theory, and an important interpolation theorem.  Then in Section \ref{thm-proof}, we give the proof of our main result. 

\subsection{Acknowledgment} X.Z. acknowledges the support by NSF grant DMS-2506717 and and a grant from the Simons Foundation. 

\section{Preliminaries} 
\label{preliminary}

In this section, we review the main tools that will be used in the proof.  First,  we discuss the space of embedded two-spheres in $S^3$.  Then we summarize some definitions from geometric measure theory.  Next we recall the Simon-Smith min-max theory for the prescribed mean curvature functional. Finally, we prove an interpolation theorem that follows by combining the Smale conjecture with a filigree retraction argument of Ketover-Liokumovich \cite{ketover2025smale}.

\subsection{The Space of Embedded Two-Spheres} Let $S^3$ be the unit three-sphere.  We denote by $\mathcal E$ the space of all oriented, smoothly embedded two-spheres in $S^3$, equipped with the smooth topology.  
We will use the following well-known consequence of the Smale conjecture \cite{hatcher1983proof}:

\begin{theorem}
\label{embedding-structure}
The space $\mathcal E$ of all oriented, smoothly embedded two-spheres deformation retracts to the space of all oriented equatorial two-spheres. In particular, $\mathcal E$ is homotopy equivalent to $S^3$.  
\end{theorem}

This has the following immediate corollary.

\begin{corollary}
\label{homology} 
The set of all oriented, equatorial two-spheres represents the non-trivial homology class in $H_3(\mathcal E,\Z_2)$. 
\end{corollary}

Given an element $\Sigma \in \mathcal E$, there is a unique open set $\Omega \subset S^3$ such that $\Sigma = \bd \Omega$ and the normal vector to $\Sigma$ determined by the orientation points into $\Omega$.  Because the functional we use depends on $\Omega$, we will frequently use the notation $(\Sigma,\Omega)$ to refer to an element of $\mathcal E$. 

\begin{defn}
Given a smooth function $h\colon S^3\to \R$, the $\mathcal A^h$ functional is defined on $\mathcal E$ by 
\[
\mathcal A^h(\Sigma,\Omega) = \area(\Sigma) - \int_\Omega h. 
\]
Smooth critical points of the $\mathcal A^h$ functional have prescribed mean curvature $h$. 
\end{defn}

The following map $\Phi\colon S^3\times (0,\pi)\to \mathcal E$ will play an important role in the proof.   For each fixed $p$, the map $r\mapsto \Phi(p,r)$ is a sweepout of $S^3$ by two-spheres which originate near $p$ and terminate near the antipodal point $-p$.  The two-spheres are oriented by the normal vector pointing to the side containing $p$. 

\begin{defn}Define $\Phi\colon S^3\times (0,\pi) \to \mathcal E$ by setting
$
\Phi(p,r) = (\bd B_p(r),B_p(r)).  
$
\end{defn}

For technical reasons, we will also work with the space of (parameterized) smooth embeddings $\operatorname{Emb}(S^2,S^3)$.  There is a projection map $\pi\colon \operatorname{Emb}(S^2,S^3)\to \mathcal E$ associating to each embedding $\phi$ its image $\phi(S^2)$.  It will be convenient for us later to know that the above map $\Phi$ admits a lift to the space of embeddings.

\begin{prop}
\label{lift}
There exists a lift $\tilde \Phi \colon S^3\times (0,\pi)\to \operatorname{Emb}(S^2,S^3)$ such that $\pi\circ \tilde \Phi = \Phi$.
\end{prop}

\begin{proof}
It suffices to construct a smooth map
\[
\tilde \Phi \colon S^3\times (0,\pi)\times S^2\to S^3
\]
with the property that $\tilde \Phi(p,r,S^2) = \Phi(p,r)$ for all $(p,r)\in S^3\times (0,\pi)$.  This can be neatly accomplished using quaternions.  Think of $\R^4$ as the space of quaternions, $S^3$ as the space of unit quaternions, and $S^2$ as the space of unit imaginary quaternions. 

We first define 
\[
\tilde \Phi\left(p,\frac{\pi}{2},q \right) = pq,
\]
where on the product on the right hand side denotes  quaternion multiplication. Observe that 
\[
\la pq,p\ra = \operatorname{Re}(pq\bar p) = \operatorname{Re}(q) = 0.
\]
Here we have used the following two facts.  First, for two quaternions $a$ and $b$, the inner product in $\R^4$ can be represented by 
\[
\la a,b\ra = \operatorname{Re}(a \bar b).
\]
Second, if $a$ is a unit quaternion then
\[
\operatorname{Re}(ab\bar a) = \frac 1 2 \left[ab\bar a + \overline{(ab\bar a)} \right] = \frac{1}{2}a\left[b + \bar b\right] \bar a = \operatorname{Re}(b) a\bar a=  \operatorname{Re}(b).
\]
It follows that $\tilde \Phi(p,\frac{\pi}{2},\cdot)$ parameterizes the great sphere obtained by intersecting $S^3$ with the plane orthogonal to $p$.  

Orient $\R^4$ so that $\{1,i,j,k\}$ is a positive frame.  Orient $S^3$ by the outward normal vector, so that a frame $\{v_1,v_2,v_3\}$ at $p$ is positively oriented if and only if $\{v_1,v_2,v_3,p\}$ is a positively oriented frame in $\R^4$.  Likewise, orient $S^2 \subset S^3$ by the normal vector $1$, so that a frame $\{v_1,v_2\}$ at $q$ is positively oriented if and only if $\{v_1,v_2,1\}$ is a positively oriented frame in $S^3$. Fix a point $p\in S^3$ and let 
\[
\Sigma_p = \tilde \Phi\left(p,\frac \pi 2, S^2\right).
\]
We claim that $\Sigma_p$ is oriented by the normal vector $p$.  
Observe that the differential of $\tilde \Phi$ is simply multiplication by $p$ on tangent vectors to $S^2$.  The differential sends the positive frame $\{i,j\}$ at the point $k \in S^2$ to the frame $\{pi,pj\}$ at the point $pk\in \Sigma_p$.  We need to check that $\{pi,pj,p,pk\}$ is a positively oriented frame in $\R^4$.  This follows from the fact that multiplication by $p$ is an orientation preserving linear map from $\R^4$ to $\R^4$.  This proves the claim.  In particular, the normal vector to $\Sigma_p$ determined by the orientation points into the region in $S^3$ containing $p$. 

Finally, we can extend to other values of $r$ by setting
\[
\tilde \Phi\left(p,\frac{\pi}{2}+t,\xi\right) = \exp_{p\xi}(-t\nu), \quad t\in\left(-\frac \pi 2, \frac \pi 2\right)
\]
where $\exp$ is the Riemannian exponential map of the round metric and $\nu$ is the unit normal to $\tilde \Phi(p,\frac{\pi}{2},S^2)$ determined by the orientation.  It is straightforward to verify that $\tilde \Phi$ is as required. 
\end{proof}

\subsection{Geometric measure theory} Let $S^3$ be the unit three-sphere. We will use the following notation from geometric measure theory. 
\begin{itemize}
\item Let $\mathcal C(S^3)$ be the set of Caccioppoli sets in $S^3$.
\item The flat topology on $\mathcal C(S^3)$ is given by $\mathcal F(\Omega_1,\Omega_2) = \vol(\Omega_1 \operatorname{\Delta} \Omega_2)$. 
\item Let $\mathcal V(S^3)$ be the set of 2-dimensional varifolds in $S^3$. 
\item Let $\mathbf F$ denote Pitts' $\mathbf F$ metric on the space of varifolds. 
\item Let $\mathcal I(S^3,\Z_2)$ denote the space of  two-dimension integral currents with $\Z_2$-coefficients in $S^3$. 
\item Given $\Omega\in \mathcal C(S^3)$, let $\bd \Omega$ denote the current induced by $\bd \Omega$. 
\item Given $T \in I(S^3,\Z_2)$, let $\vert T\vert$ denote the varifold induced by $T$. 
\item The space $\vc(S^3)$ is the set of all pairs $(V,\Omega) \in \mathcal V(S^3)\times \mathcal C(S^3)$ such that there is a sequence $\{\Omega_i\}$ in $\mathcal C(S^3)$ with $\Omega_i \to \Omega$ and $\vert \bd \Omega_i\vert \to V$. 
\item The $\mathscr F$-topology on $\vc(S^3)$ is given by 
\[
\mathscr F \big((V_1,\Omega_1),(V_2,\Omega_2)\big) = \mathbf F(V_1,V_2) + \mathcal F(\Omega_1,\Omega_2). 
\]
\end{itemize}
The $\vc$ space has been used before in \cite{mazurowski2024infinitely} and \cite{wang2023existence}  and is very convenient for the prescribed mean curvature min-max theory. 

\subsection{Almost-Embeddings} 
Next, we review some definitions regarding almost-embedded surfaces. We follow the definitions in \cite{wang2023existence}.  Let $\Sigma$ be a closed surface. 

\begin{defn}
We say that a $C^{1,1}$ immersion $\phi\colon \Sigma \to S^3$ is an almost-embedding if near each point $p\in \phi(\Sigma)$, there is a neighborhood $U$ of $p$ in $S^3$ such that 
\begin{itemize}
\item[(i)] $\phi\inv (U)$ is a disjoint union of connected components $\Gamma_1,\hdots,\Gamma_\ell$,
\item[(ii)] $\phi|_{\Gamma_i}$ is a $C^{1,1}$ embedding for each $i$, and 
\item[(iii)] for each $i\neq j$, the sheet $\phi_i(\Gamma_i)$ lies to one side of $\phi_j(\Gamma_j)$ in $U$. 
\end{itemize} 
In other words, near each point $p$ the surface $\phi(\Sigma)$ can be decomposed into an ordered union of embedded sheets.  
\end{defn} 

We will sometimes abuse notation by writing $\Sigma$ for $\phi(\Sigma)$.  Define the regular set $\mathcal R(\Sigma)$ to be the set of points $p\in \phi(\Sigma)$ where $\Sigma$ is embedded, and define the touching set $\mathcal S(\Sigma)$ to be the set of points $p\in \phi(\Sigma)$ that belong to more than one sheet of $\Sigma$. 

\begin{defn}
A $C^{1,1}$ almost embedding $\phi\colon \Sigma\to S^3$ is called a $C^{1,1}$ boundary if there is a set $\Omega_\Sigma\in \mathcal C(S^3)$ such that 
\[
\phi_\sharp \llbracket \Sigma \rrbracket = \bd \Omega_\Sigma.
\]
Here $ \llbracket \Sigma \rrbracket$ denotes the fundamental class of $\Sigma$ and the equality is in the sense of two-dimensional currents in $S^3$. 
\end{defn}

We note that any $C^{1,1}$ boundary $\phi \colon \Sigma\to S^3$ determines an element $(V_\Sigma,\Omega_\Sigma) \in \vc(S^3)$, which by abuse of notation we will usually denote by $(\Sigma,\Omega)$.  Given a smooth function $h$ on $S^3$, the $\mathcal A^h$ functional can be naturally extended to any $C^{1,1}$ boundary by setting 
\[
\mathcal A^h(\Sigma,\Omega) = \mathcal A^h(V_\Sigma,\Omega_\Sigma) = \|V_\Sigma\|(S^3) - \int_{\Omega_\Sigma} h.
\]
Finally, we have the following definition regarding $C^{1,1}$ boundaries which are critical points of the $\mathcal A^h$ functional. 

\begin{defn}
Let $h$ be a smooth function on $S^3$.  A $C^{1,1}$ boundary $\phi \colon \Sigma \to S^3$ is called a $C^{1,1}$ $h$-boundary if $\delta \mathcal A^h_{(\Sigma,\Omega)} (X) = 0$ for all $C^1$ vector fields $X$ on $S^3$. 
\end{defn}

\subsection{Min-Max Theory} 
Next we recall some min-max notions. The original Simon-Smith min-max theory for the area functional appeared in Smith's doctoral thesis \cite{smith1982existence} (see also \cite{Colding-DeLellis03}), and later De Lellis-Pellandini \cite{DeLellis-Pellandini10} and Ketover \cite{ketover2019genus} proved a genus upper bound for the min-max minimal surfaces obtained via this theory. The adaptation to the prescribed mean curvature setting is due to Sarnataro-Stryker \cite{sarnataro2023optimal}, with further refinements by Z. Wang and the second author \cite{wang2023existence}.  Again, we will follow the notation from \cite{wang2023existence}.  

\begin{defn}
We say that a smooth function $h\colon S^3\to \R$ is admissible for the PMC min-max theory if $h$ belongs to the class $\mathcal S$ from  \cite[Proposition 0.2]{zhou2020existence} or if it satisfies condition ($\ddagger$) from \cite{zhou2020existence}.  Every positive function is automatically admissible, and the set of admissible functions is generic (in the sense of Baire category).  We note that for this definition of admissibility, no minimal surface can be contained in the zero set of $h$.
\end{defn}

Let $X$ be a cubical complex and let $Z$ be a subcomplex of $X$.  Fix a continuous map $\Phi\colon X\to \mathcal E$.  Let $\Pi$ be the set of all continuous maps $\Psi\colon X\to \mathcal E$ which are of the form 
\[
\Psi(x) = [\psi(x,1)](\Phi(x))
\]
where $\psi\colon X\times [0,1] \to \operatorname{Diff}(S^3)$ is a continuous map satisfying 
\begin{itemize}
    \item [(i)] $\psi(x,0) = \operatorname{id}$ for all $x\in X$, and 
    \item[(ii)] $\psi(z,t) = \operatorname{id}$ for all $z\in Z$ and all $t\in [0,1]$. 
    \end{itemize} 
We will refer to $\Pi$ as the $(X,Z)$-relative homotopy class of $\Phi$.  Note that if $\Phi$ admits a lift to a map into $\operatorname{Emb}(S^2,S^3)$ then every map $\Psi\in \Pi$ also admits such a lift. 

Fix a smooth function $h\colon S^3\to \R$ which is admissible for the PMC min-max theory.  Define the min-max value 
\[
L(\Pi) = \inf_{\Psi\in \Pi} \left[\sup_{x\in X} \mathcal A^h(\Psi(x))\right]. 
\]
A sequence $\{\Psi_i\}$ in $\Pi$ is called a critical sequence if 
\[
\lim_{i\to \infty} \sup_{x\in X} \Psi_i(x) = L(\Pi). 
\]
Further, given a subsequence $i_j$ and elements $x_{i_j}\in X$, we say that $\{\Psi_{i_j}(x_{i_j})\}$ is a min-max sequence if 
\[
\lim_{j\to \infty} \mathcal A^h(\Psi_{i_j}(x_{i_j})) = L(\Pi). 
\]
Finally, given a critical sequence $\{\Psi_i\}$, define the critical set 
\[
\mathcal K(\{\Psi_i\}) = \left\{(V,\Omega) \in \vc(S^3):  \begin{array}{c} \text{there is a min-max sequence $\{\Psi_{i_j}(x_{i_j})\}$} \\ \text{such that $\mathscr F(\Psi_{i_j}{(x_{i_j})}, (V,\Omega)) \to 0$.} \end{array}\right\}.
\]
Then we have the following min-max theorem \cite{wang2023existence}.  
\begin{theorem}
\label{min-max}
Let $h\colon S^3\to \R$ be a smooth function which is admissible for the PMC min-max theory. Let $(X,Z)$ be a pair as above.  Fix a continuous map $\Phi\colon X \to \mathcal E$ and define $\Pi$ as above.  Assume that 
\[
L(\Pi) > \max\left\{\sup_{z\in Z} \mathcal A^h(\Phi(z)), 0\right\}.  
\]
Let $\{\Psi_i\}$ be a critical sequence.  Then some element $(\Sigma,\Omega)$ in the critical set is a $C^{1,1}$ h-boundary with $\mathcal A^h(\Sigma,\Omega) = L(\Pi)$.  
\end{theorem}

It follows from \cite{sarnataro2023optimal} that the regularity can be improved at points of density one; also see \cite[Lemma 1.14]{wang2023existence}.

\begin{theorem}
\label{regularity} 
Let $(\Sigma,\Omega)$ be the pair produced by Theorem \ref{min-max}.  If $\Sigma$ has no touching points, then $\Sigma$ is a disjoint union of smoothly embedded two-spheres with prescribed mean curvature $h$. \end{theorem} 

\subsection{An Interpolation Theorem} In this section, we prove an interpolation theorem.  The interpolation theorem has two key ingredients.  The first is the following equivalent form of the Smale conjecture; see \cite[Page 430]{hatcher1981diffeomorphism} assertion (a) and also item (10) in the appendix of \cite{hatcher1983proof}. 

\begin{theorem}
\label{smales2}
The space of essential embeddings $\operatorname{Emb}_0(S^2,S^2\times (-1,1))$ deformation retracts onto the subspace of diffeomorphisms $S^2 \to S^2\times \{0\}$. 
\end{theorem}

The second main ingredient is the following filigree retraction procedure devised by Ketover-Liokumovich \cite[Proposition 6.1]{ketover2025smale}.  We will state a less general version which suffices for our purposes.  The filigree retraction argument uses embeddings, which explains our need to consider parameterized two-spheres.   Given $\Sigma \in \mathcal E$, let $N_r(\Sigma)$ denote the $r$-tubular neighborhood of $\Sigma$. Conclusion (iii) of the following proposition is not explicitly stated in \cite{ketover2025smale}, but it follows from the proof.  

\begin{prop}[\cite{ketover2025smale}]
\label{filigree}
Fix some $\Sigma \in \mathcal E$.    Then for any sufficiently small $r$, there exists an $\eps > 0$ so that the following is true.  Let $X$ be a three dimensional cubical complex with no boundary.  For any continuous map $f \colon X \to \operatorname{Emb}(S^2,S^3)$ satisfying 
\begin{itemize}
\item[(i)] $\mathbf F(f(x),\Sigma) < \eps$ for all $x\in X$,
\end{itemize}
there exists a homotopy $F\colon X \times [0,1] \to \operatorname{Emb}(S^2,S^3)$ such that 
\begin{itemize}
\item[(i)] $F(x,0) = f(x)$ for all $x\in X$,
\item[(ii)] $F(x,1)$ is contained in $N_r(\Sigma)$ for all $x\in X$,
\item[(iii)] $F(x,1)$ does not bound a ball in $N_r(\Sigma)$ for all $x\in X$. 
\end{itemize}
\end{prop}

We now combine these ingredients to prove the interpolation theorem.

\begin{theorem}
\label{interpolation}
Fix some $\Sigma\in \mathcal E$.  There is an $\eps(\Sigma) > 0$ so that the following is true.  Let $X$ be a three dimensional cubical complex with no boundary.  Assume $f\colon X\to \mathcal E$ is a continuous map satisfying 
\begin{itemize}
\item[(i)] $\mathbf F(f(x),\Sigma)<\eps$ for all $x\in X$,
\item[(ii)] $f$ admits a lift to a map $\tilde f \colon X \to \operatorname{Emb}(S^2,S^3)$. 
\end{itemize}
Then there exists a continuous map $F\colon X\times [0,1]\to \mathcal E$ such that 
\begin{itemize}
\item[(i)] $F(x,0) = f(x)$ for all $x\in X$,
\item[(ii)] $F(x,1) = \Sigma$ for all $x\in X$. 
\end{itemize}
\end{theorem}

\begin{proof}
Fix a small number $r > 0$ for which $N_r(\Sigma)$ is diffeomorphic to $S^2\times (-1,1)$.  Assuming $\eps$ is small enough, we can apply Proposition \ref{filigree} to $\tilde f$ to obtain a homotopy $H_1\colon X\times [0,1]\to \operatorname{Emb}(S^2,S^3)$ satisfying 
\begin{itemize}
\item[(i)] $H_1(x,0)=\tilde f(x)$ for all $x\in X$,
\item[(ii)] $H_1(x,1)$ is contained in $N_r(\Sigma)$ for all $x\in X$,
\item[(iii)] $H_1(x,1)$ does not bound a ball in $N_r(\Sigma)$ for all $x\in X$. 
\end{itemize}
Since each embedding $H_1(\cdot,1)$ has image in $N_r(\Sigma)$, we can now apply Theorem \ref{smales2} to find a second homotopy 
$H_2\colon X\times [0,1] \to \operatorname{Emb}(S^2,N_r(\Sigma))$ such that 
\begin{itemize}
\item[(i)] $H_2(x,0)= H_1(x,1)$ for all $x\in X$,
\item[(ii)] $ H_2(x,1)$ has image $\Sigma$ for all $x\in X$. 
\end{itemize}
Finally define 
\[
F(x,t) = \begin{cases} 
\pi\circ  H_1(x,2t), &\text{if } t\in [0,\frac 1 2]\\
\pi \circ  H_2(x,2t-1), &\text{if } t \in [\frac 1 2, 1].
\end{cases}
\]
It is easy to see that $F$ is as required. 
\end{proof}

\section{Proof of Main Result} 
\label{thm-proof} 
In this section, we prove the main result Theorem \ref{main}.  Fix a smooth function $h\colon S^3\to \R$ which is admissible for the PMC min-max theory and assume that $\vert h\vert < h_0$.  By replacing $h$ with $-h$ if necessary, we can assume that 
\begin{equation}
\label{positive-integral} 
\int_{S^3} h \ge 0. 
\end{equation}
The proof is by contradiction.  Assume that there is only one smooth, oriented, embedded two-sphere in $S^3$ with mean curvature given by $h$.  Denote this two-sphere  by $(\Gamma,\Theta)\in \mathcal E$.

\subsection{One-Parameter Min-Max}
In the first step, we set up a collection of 1-parameter min-max problems, and observe that each of these must detect $(\Gamma,\Theta)$.  Fix a very small number $\delta > 0$ to be specified later.  For each $p\in S^3$, define a map $\Phi_p\colon [\delta,\pi-\delta]\to \mathcal E$ by setting 
\[
\Phi_p(r) = (\bd B_p(r),B_p(r)). 
\]
Let $\Pi_p$ denote the $([\delta,\pi-\delta],\{\delta,\pi-\delta\})$-relative homotopy class of $\Phi_p$.  Thus $\Pi_p$ consists of sweepouts that originate at a small sphere near $p$ and terminate in a small sphere near the antipodal point $-p$. 

\begin{prop}
\label{width-bound}
The min-max value $L(\Pi_p)$ satisfies 
\[
\max\bigg\{\mathcal A^h(\bd B_p(\delta), B_p(\delta)), \mathcal A^h(\bd B_p(\pi-\delta), B_p(\pi-\delta)), 0\bigg\} < L(\Pi_p) < 4\pi + 2\pi^2 h_0. 
\]
\end{prop}

\begin{proof}
First we concentrate on the lower bound. Using \eqref{positive-integral} and the fact that $\vert h\vert < h_0$, we see that  
\[
\max\bigg\{\mathcal A^h(\bd B_p(\delta), B_p(\delta)), \mathcal A^h(\bd B_p(\pi-\delta), B_p(\pi-\delta)), 0\bigg\} \le \area(\bd B_p(\delta)) + h_0 \vol(B_p(\delta)). 
\]
We can make the right hand side arbitrarily close to 0 by choosing $\delta$ small enough.  Now consider any map $\Psi \in \Pi_p$.  There must be some $t$ for which $\Psi(t) = (\Sigma,\Omega)$ satisfies $\vol(\Omega) = \frac{1}{2} \vol(S^3) = \pi^2$.  By the isoperimetric inequality on $S^3$, this implies that 
\[
\area(\Sigma) \ge 4\pi. 
\]
It follows that 
\[
\mathcal A^h(\Sigma,\Omega) = \area(\Sigma) - \int_\Omega h \ge 4\pi - h_0 \pi^2 > 0.
\]
This proves the lower bound. 

To prove the upper bound, it suffices to show that 
\[
\sup_{r\in [\delta,\pi-\delta]} \mathcal A^h(\Phi_p(r)) < 4\pi + 2 \pi^2 h_0. 
\]
To see this, note that for any $r \in [\delta,\pi-\delta]$ we have 
\[
\mathcal A^h(\bd B_p(r), B_p(r)) = \area(\bd B_p(r)) - \int_{B_p(r)} h < 4\pi + h_0\vol(S^3) = 4\pi + 2 \pi^2 h_0. 
\]
The result follows. 
\end{proof}

We will also need the following result from \cite{sarnataro2023optimal}, which exploits the monotonicity formula to convert a lower bound on density into a lower bound on area.  

\begin{prop}[\cite{sarnataro2023optimal} Proposition 19.3]
\label{density} Let $S^3$ be the unit 3-sphere.  Assume that $V$ is a two-dimensional varifold in $S^3$ with $h_0$-bounded first variation.  If the density of $V$ satisfies $\Theta^2(V,x) \ge 2$ for some $x\in S^3$ then 
\[
\|V\|(S^3) \ge \frac{8\pi}{1 + \frac{h_0^2}{4}} = 4\pi + 2\pi^2 h_0, 
\]
where the last equality uses the definition \eqref{h0def} of $h_0$. 
\end{prop}

Applying Theorem \ref{min-max} to the homotopy class $\Pi_p$, there is a $C^{1,1}$ $h$-boundary $(\Sigma_p,\Omega_p)$ satisfying $\mathcal A^h(\Sigma_p,\Omega_p) = L(\Pi_p)$.  Since $L(\Pi_p)< 4\pi + 2\pi^2 h_0$, it follows from Proposition \ref{density} that the density of $\Sigma_p$ is one everywhere.  In particular, $\Sigma_p$ has no touching points, and so Theorem \ref{regularity} implies that $\Sigma_p$ is a disjoint union of embedded two-spheres with prescribed mean curvature $h$.  Thus $(\Sigma_p,\Omega_p) = (\Gamma,\Theta)$ and it follows that $
L(\Pi_p) = \mathcal A^h(\Gamma,\Theta)$
for every $p\in S^3$.  To summarize, we have the following result. 

\begin{prop}
\label{one-parameter}
    For every $p\in S^3$, we have 
    $
    L(\Pi_p) = \mathcal A^h(\Gamma,\Theta).
    $
    Moreover, for any critical sequence $\{\Psi_i\}$ for $\Pi_p$, the pair $(\Gamma,\Theta)$ belongs to the critical set $\mathcal K(\{\Psi_i\})$.  
\end{prop}

\subsection{Four-Parameter Min-Max} In the next step, we consider a 4-parameter min-max problem.  Define $X = S^3\times [\delta,\pi-\delta]$ and $Z = S^3\times \{\delta,\pi-\delta\}$.  Then define a map $\Phi\colon X\to \mathcal E$ by setting 
\[
\Phi(p,r) = \Phi_p(r) = (\bd B_p(r), B_p(r)). 
\] 
The map $\Phi$ simply collects all of the above sweepouts into a single family.  Let $\Pi$ be the $(X,Z)$-relative homotopy class of $\Phi$. From Proposition \ref{lift}, we know that $\Phi$ admits a lift to a map into the space of embeddings $\operatorname{Emb}(S^2,S^3)$.  Therefore, every map $\Psi\in \Pi$ also admits such a lift. 

\begin{prop}
The min-max value $L(\Pi)$ satisfies 
\[
\max\bigg\{\sup_{z\in Z} \mathcal A^h(\Phi(z)), 0 \bigg\} < L(\Pi) < 4\pi + 2 \pi^2 h_0. 
\]
\end{prop}

\begin{proof}
First note that 
\[
\sup_{z\in Z} \mathcal A^h(\Phi(z)) = \sup_{p\in S^3} \max \left\{\mathcal A^h(\bd B_p(\delta), B_p(\delta), \mathcal A^h(\bd B_p(\pi-\delta), B_p(\pi-\delta))\right\},
\]
and that this can be made arbitrarily close to 0 by choosing $\delta$ small enough. Since 
\[
L(\Pi) \ge L(\Pi_p) \ge 4\pi - h_0 {\pi^2} > 0,
\]
we obtain the lower bound. The upper bound follows exactly as in Proposition \ref{width-bound} since every element in the image of $\Phi$ is of the form $(\bd B_p(r), B_p(r))$ for some $p\in S^3$ and $r \in [\delta,\pi-\delta]$. 
\end{proof} 

We now apply Theorem \ref{min-max} to the homotopy class $\Pi$.  This produces a $C^{1,1}$ $h$-boundary $(\Sigma,\Omega)$ satisfying $\mathcal A^h(\Sigma,\Omega) = L(\Pi)$.  Again, since $L(\Pi) < 4\pi + 2 \pi^2 h_0$, it follows from Proposition \ref{density} that $\Sigma$ has no touching points, and hence $\Sigma$ is a disjoint union of smooth, embedded two-spheres with mean curvature given by $h$.  Thus we have $(\Sigma,\Omega) = (\Gamma,\Theta)$ and it follows that $L(\Pi) = \mathcal A^h(\Gamma,\Theta)$.  

Finally, we are going to use the fact that $L(\Pi) = L(\Pi_p)$ for all $p\in S^3$ to deduce a contradiction.   Choose a critical sequence $\{\Psi_i\}$ for $\Pi$.   
Let $\eps(\Gamma) > 0$ be the number associated to $\Gamma$ by Theorem \ref{interpolation}.  The following proposition can be proved by following an argument of Marques-Neves \cite{marques2014min} almost verbatim.  We provide an alternative proof that elucidates the connection with the cup product argument in Lusternik-Schnirelmann theory. 

\begin{prop} 
\label{LS} We can find a large number $i$ and a subcomplex $Y$ of $X$ with no boundary satisfying the following properties: 
\begin{itemize}
\item[(i)] $[Y] = [S^3 \times \{\frac{\pi}{2}\}]$ in $H_3(X,\Z_2)$,
\item[(ii)] for every $y\in Y$, we have $\mathbf F(\Psi_i(y),\Gamma) < \eps(\Gamma)$.
\end{itemize}
\end{prop}

\begin{proof}
Choose a small number $\eps < \frac{\eps(\Gamma)}{2}$.  Then for each $i$, let $X(\ell_i)$ be a very fine refinement of $X$.   Let $\mathcal T_i$ be the union of the closed 4-dimensional cells $c$ in $X(\ell_i)$ for which there is a point $x  \in c$ with $\mathbf F(\Psi_i(x),\Gamma) \le \eps$.  Assuming $\ell_i$ is large enough, we can ensure that $\mathbf F(\Psi_i(y),\Gamma) \le 2\eps$ for all $y$ in such a cell $c$.  Thus we have 
\[
\mathbf F(\Psi_i(y),\Gamma)\le 2\eps < \eps(\Gamma)
\]
 for all $y\in \mathcal T_i$.   

We will use homology and cohomology with $\Z_2$ coefficients throughout the proof.  We define two open subsets of $X$.  First, since $\mathcal T_i$ is a subcomplex, we can find a small neighborhood $U_i$ of $\mathcal T_i$ which deformation retracts to $\mathcal T_i$.   Second, define $W_i = X\setminus \mathcal{T}_i$.  By choosing $\ell_i$ sufficiently large, it is possible to ensure that $U_i \cap Z = \emptyset$ and $Z \subset W_i$. 
To prove the proposition, it suffices to show there is a homogeneous 3-dimensional subcomplex $Y$ of $\mathcal T_i$ with no boundary such that $[Y] = [S^3\times \{\frac{\pi}{2}\}]$ in $H_3(X)$. 

In other words, we want to show that $[S^3\times \{\frac{\pi}{2}\}]$ is in the image of the natural map $H_3(\mathcal T_i) \to H_3(X)$.  It is equivalent to show that $[S^3\times \{\frac{\pi}{2}\}]$ is in the image of the map $H_3(U_i) \to H_3(X)$.  Suppose for contradiction that this is not the case. Then by the long exact sequence in homology, $[S^3\times \{\frac{\pi}{2}\}]$ maps to a non-zero class $u = j_*[S^3\times \{\frac{\pi}{2}\}]$ in $H_3(X, U_i)$.   Since we are using $\Z_2$ coefficients, we have $H^3(X,U_i)\cong \operatorname{Hom}(H_3(X,U_i),\Z_2)$.  Choose a basis for $H_3(X,U_i)$ which includes $u$.  Then define $\alpha_{U} \in H^3(X,U_i)$ to be the element that evaluates to 1 on $u$ and 0 on all other members of the basis.  Now consider the maps 
\begin{gather*}
j_*\colon H_3(X) \to H_3(X,U_i),\\
j^*\colon H^3(X,U_i) \to H^3(X). 
\end{gather*}
Observe that 
\[
j^*\alpha_U\left(\left[S^3\times \left\{\frac{\pi}{2}\right \}\right ]\right ) = \alpha_U\left(j_*\left[S^3\times \left\{\frac{\pi}{2}\right \}\right ]\right) = 1. 
\]
Therefore $j^*\alpha_U$ is the non-zero element $\alpha \in H^3(X)$. 

Now let $I = [\delta,\pi-\delta]$ so that $\bd I = \{\delta,\pi-\delta\}$.  Fix any point $p\in S^3$.  Since $Z\subset W_i$, we can consider the class $[\{p\}\times \bd I]\in H_0(W_i)$. We claim that this class is non-zero for sufficiently large $i$.  Indeed, if this is not the case, then for a (not relabeled) subsequence of $i$'s going to infinity we can find a continuous path $\gamma_i\colon [\delta,\pi-\delta] \to W_i$ with $\gamma_i(\delta) = (p,\delta)$ and $\gamma_i(\pi-\delta) = (p,\pi-\delta)$.  The fact that $L(\Pi) = L(\Pi_p)$ implies that the sequence $\{\Psi_i\circ \gamma_i\}$ is a critical sequence for $\Pi_p$.  Therefore, by Proposition \ref{one-parameter}, there must be a sequence of times $t_i$ such that $\Psi_i(\gamma_i(t_i))$ converges in the $\mathbf F$ metric to $\Gamma$. Since $W_i = X\setminus \mathcal T_i$, this contradicts the definition of $\mathcal T_i$, and the claim follows. 

Now we argue similarly to above.  Observe that $[\{p\}\times I]$ maps to $[\{p\}\times \bd I]$ under the boundary homomorphism $H_1(X,W_i)\to H_0(W_i)$.  Since $[\{p\}\times \bd I]$ is non-zero in $H_0(W_i)$, it follows that $[\{p\}\times I]$ is non-zero in $H_1(X,W_i)$.  Choose a basis for $H_1(X,W_i)$ which includes $[\{p\}\times I]$, and then define $\beta_W \in H^1(X,W_i)$ to be the class which evaluates to $1$ on $[\{p\}\times I]$ and zero on all other elements of the basis. Since $Z \subset W_i$, we can consider the maps 
\begin{gather*}
j_*\colon H_1(X,Z) \to H_1(X,W_i),\\
j^*\colon H^1(X,W_i) \to H^1(X,Z). 
\end{gather*}
Observe that 
\[
j^*\beta_W([\{p\}\times I]) = \beta_W(j_*[\{p\}\times I]) = 1
\]
and so $j^*\beta_W$ is the non-zero element $\beta\in H^1(X,Z)$. 

Finally, observe that the relative cup product 
\[
\alpha_U \smile \beta_W \in H^4(X,U_i\cup W_i)
\]
must be zero since $U_i\cup W_i = X$.  On the other hand, by naturality, $\alpha_U \smile \beta_W$ maps to $\alpha \smile \beta$ under the map 
\[
H^4(X,X) \to H^4(X,Z). 
\]
Therefore $\alpha \smile \beta$ is zero.  But now observe that $\alpha \in H^3(X)$ is the Lefschetz dual of $[\{p\}\times I] \in H_1(X,Z)$ and $\beta \in H^1(X,Z)$ is the Lefschetz dual of $[S^3\times \{\frac \pi 2\}]\in H_3(X)$.  It is well-known that the cup product is dual to intersection.  Since $\{p\}\times I$ and $S^3\times \{\frac \pi 2\}$ intersect transversally in a single point, it follows that $\alpha \smile \beta \in H^4(X,Z)$ is non-zero.  This is a contradiction, and the proof is complete.
\end{proof} 

Finally, we can conclude the proof of Theorem \ref{main}. Choose a large number $i$ and a subcomplex $Y$ of $X$ satisfying the conclusion of Proposition \ref{LS}.  According to Theorem \ref{interpolation}, there is a homotopy $H\colon Y\times [0,1] \to \mathcal E$ such that 
\begin{itemize}
\item[(i)] $H(y,0) = \Psi_i(y)$, and
\item[(ii)] $H(y,1) = \Gamma$, 
\end{itemize}
for all $y\in Y$. 
The existence of such a homotopy implies that $(\Psi_i)_* [Y] = 0$ in $H_3(\mathcal E,\Z_2)$.  On the other hand, we have $[Y] = [S^3\times \{\frac{\pi}{2}\}]$ in $H_3(X,\Z_2)$ and so 
\[
\Phi_*\left[S^3 \times \left\{\frac {\pi}{2}\right\}\right]  =  (\Psi_i)_*\left[S^3\times \left\{\frac{\pi}{2}\right\}\right] = (\Psi_i)_*[Y] = 0
\]
since $\Psi_i$ is homotopic to $\Phi$.  But $\Phi_*[S^3 \times \{\frac \pi 2\}]$ represents the non-trivial element in $H_3(\mathcal E,\Z_2)$ by Corollary \ref{homology}.  This is a contradiction, and the result follows. 

\bibliographystyle{plain}
\bibliography{biblio.bib}

\end{document}